\documentclass[reqno]{amsart}
\usepackage{hyperref}

\AtBeginDocument{{\noindent\small
 }
\vspace{9mm}}

\begin{document}
\title[ Some inverse problems]
{Some inverse problems associated with Hill operator}

\author[A. A. K\i ra\c{c} ]
{Alp Arslan K\i ra\c{c}}


\address{Alp Arslan K\i ra\c{c} \newline
Department of Mathematics, Faculty of Arts and Sciences, Pamukkale
University, \newline
 20070, Denizli, Turkey}
\email{aakirac@pau.edu.tr}

\keywords{Hill operator; inverse spectral theory; eigenvalue asymptotics; \hfill\break\indent
 Fourier coefficients}

\begin{abstract}
 Let $l_{n}$ be the length of the $n$-th instability interval of the Hill operator $Ly=-y^{\prime\prime}+q(x)y$.
We obtain that if $l_{n}=o(n^{-2})$ then $c_{n}=o(n^{-2})$, where $c_{n}$ are the Fourier coefficients of $q$. Using this inverse result, we prove:  Let $l_{n}=o(n^{-2})$. If  $\{(n\pi)^{2}: \textrm{$n$ even and $n>n_{0}$}\}$ is a subset of the periodic spectrum of Hill operator then $q=0$ a.e., where $n_{0}$ is a positive large number such that $l_{n}<\varepsilon n^{-2}$ for all $n>n_{0}(\varepsilon)$ with some $\varepsilon>0$.
A similar result holds for the anti-periodic case.
\end{abstract}

\maketitle
\numberwithin{equation}{section}
\newtheorem{theorem}{Theorem}[section]
\newtheorem{lemma}[theorem]{Lemma}
\newtheorem{definition}[theorem]{Definition}
\allowdisplaybreaks
\section{Introduction}

Consider the Hill  operator
\begin{equation}  \label{1}
Ly=-y^{\prime\prime}+q(x)y
\end{equation}
generated in $L_{2}(-\infty, \infty)$ , where $q(x)$ is a reel-valued summable function on $[0,1]$ and $q(x+1)=q(x)$. Let  $
\lambda_{n}$ and $\mu_{n}$ $(n=0,1,\ldots)$ denote, respectively, the
$n$-th periodic and anti-periodic eigenvalues of the Hill operator (\ref{1})
on $[0,1]$ with the periodic boundary conditions
\begin{equation}\label{per.bc}
    y(0)=y(1),\, y^{\prime}(0)=y^{\prime}(1),
\end{equation}
 and the anti-periodic boundary conditions
\[
    y(0)=-y(1),\; y^{\prime}(0)=-y^{\prime}(1).
\]
It is well-known  \cite{Coddington:Levinson,Eastham} that
\[  
\lambda_{0}<\mu_{0}\leq
\mu_{1}<\lambda_{1}\leq\lambda_{2}<\mu_{2}%
\leq \mu_{3}<\cdots \rightarrow\infty.
\]
The intervals $(\mu_{2m},\mu_{2m+1})$ and $(\lambda_{2m+1},\lambda_{2m+2})$
are respectively referred to as  the $(2m+1)$-th and $(2m+2)$-th finite instability intervals of the operator $L$, while
$(-\infty,\lambda_{0})$ is called the zero-th instability interval. The length of the $n$-th instability interval of (\ref{1})
will be denoted by $l_{n}$ ($n=2m+1,\,2m+2$). For further background see \cite{Magnus-Winkler, Marchenko, Hochstadt:determination}.

Borg \cite{Borg}, Ungar \cite{Ungar} and Hochstadt \cite{Hochstadt:determination} proved independently of each other the following statement:

\emph{If $q(x)$ is real and integrable, and if all finite instability intervals
vanish then $q(x)=0$ a.e.}

Hochstadt \cite{Hochstadt:determination} also showed that, when precisely one of the finite instability intervals does not vanish,  $q(x)$ is the elliptic function
which satisfies \[q^{\prime\prime}=3q^{2}+Aq+B \quad a.e.,\]
where $A$ and $B$ are suitable constants, and, when
$n$ finite instability intervals fail to vanish, $q(x)$ is infinitely differentiable a.e. For more results concerning the above type and further references, see \cite{Goldberg:determination, Goldberg:necessary, Goldberghoch:selected, Goldberghoch:finitenumber}.

Also, by using  length of the instability interval, let us consider another approach to inverse problems. Hochstadt \cite{Hochstadt:Stability-Estimate} proved that the lengths of the instability intervals $l_{n}$
vanish faster than any power of $(1/n)$ for an $L^{2}_{1}$ potential $q$ in $C_{1}^{\infty}$. McKean and Trubowitz \cite{McKean} proved the converse: if $q$ is in $L_{1}^{2}$
and the length of the $n$-th instability interval for $n\geq 1$ is rapidly decreasing, then $q$ is in $C_{1}^{\infty}$. Later Trubowitz \cite{Trubowitz} proved the following result: an $L_{1}^{2}$ potential $q$ is real analytic if and only if the lengths of the instability intervals are decays exponentially. In \cite{Coskun:invers}, Coskun showed that (see Theorem 6), in our notations,
\begin{equation}\label{coskun}
\hspace{-7em}\textrm{if}\,\,l_{n}=O(n^{-2})\textrm{ then}\,\, c_{n}=:(q\,,e^{i2n\pi x})=O(n^{-2})\textrm{ as $n\rightarrow \infty,$}
\end{equation}
where $(.\,,.)$ is the inner product in $L^{2}[0,1].$

At this point we refer to some Ambarzumyan-type theorems in \cite{Ambarz, corri, Ambarzcoupled, anoteinver}. In 1929, Ambarzumyan \cite{Ambarz} obtained the following first theorem in inverse spectral theory:
If $\{n^{2} : n = 0,1,\ldots\}$ is the spectrum of the Sturm-Liouville operator (\ref{1}) on $[0,1]$ with Neumann
boundary condition, then $q = 0$ a.e. In \cite{corri}, they extended the classical Ambarzumyan's theorem for the Sturm-Liouville equation to the general separated boundary
conditions,  by imposing an additional condition on the potential function, and their result supplements the P{\"{o}}schel-Trubowitz inverse spectral theory \cite{Poschel}. In \cite{Ambarzcoupled}, based on the well-known extremal property of the first eigenvalue, they find two analogs of Ambarzumyan's theorem to Sturm-Liouville systems of n dimension under periodic or anti-periodic boundary conditions. In the paper \cite {anoteinver}, by using Rayleigh-Ritz inequality and imposing a condition on the second term in the Fourier cosine series (see (\ref{cond})), they proved the following Ambarzumyan-type theorem:

\emph{(a) If all periodic eigenvalues of Hill's equation (\ref{1}) are nonnegative and they include $\{(2m\pi)^{2}: m\in \mathbb{N}\}$, then $q=0$ a.e.}

\emph{(b) If If all anti-periodic eigenvalues of Hill's equation (\ref{1}) are not less than $\pi^{2}$ and they include $\{(2m-1)^{2}\pi^{2}: m\in \mathbb{N}\}$, and
\begin{equation}  \label{cond}
\int_{0}^{1}q(x)\,cos(2\pi x)\,dx\geq 0,
\end{equation} then $q=0$ a.e.}

More recently, in \cite{kýrac:ambars}, we obtain the classical Ambarzumyan's theorem for the Sturm-Liouville
operators with $q\in L^{1}[0,1]$ and quasi-periodic boundary conditions,
when there is not any additional condition on the potential $q$ such as (\ref{cond}). See further references in \cite{kýrac:ambars}.

In this paper, we prove the following results:
\begin{theorem} \label{lemma}
If $l_{n}=o(n^{-2})$ then $c_{n}=o(n^{-2})$ as $n\rightarrow \infty.$
\end{theorem}
\begin{theorem} \label{main0}
Let $l_{n}=o(n^{-2})$ as $n\rightarrow \infty$. Then

(i) if  $\{(n\pi)^{2}: \textrm{$n$ even and $n>n_{0}$}\}$ is a subset of the periodic spectrum of Hill operator then $q=0$ a.e.

(ii) if  $\{(n\pi)^{2}: \textrm{$n$ odd and $n>n_{0}$}\}$ is a subset of the anti-periodic spectrum of Hill operator then $q=0$ a.e.,
where $n_{0}$ is a positive large number such that
\[
l_{n}<\varepsilon n^{-2}\qquad \textrm{for all $n>n_{0}(\varepsilon)$ with some $\varepsilon>0$.}
\]
\end{theorem}

In Theorem \ref{lemma}, we obtain that $O$-terms in (\ref{coskun}) can be improved to the $o$-terms $o(n^{-2})$ from which we shall use essentially in the proof of Theorem \ref{main0}.

Note that the first eigenvalue for Ambarzumyan-type theorems is important to be given  while for some of the other types the multiplicity of some eigenvalues is important, that is, some of the instability intervals vanish. Unlike the works in types briefly outlined above, to prove the assertion of Theorem \ref{main0} we use not only the length of instability internals $l_{n}$ as $n\rightarrow \infty$ but also a subset of spectrum of Hill operator as in Ambarzumyan-type theorems. However, in Theorem \ref{main0}, we assume that, for some large $n_{0}$, $(n\pi)^{2}$ with $n>n_{0}$  is a periodic eigenvalue for even $n$ (or anti-periodic for odd $n$)  and we do not assume that the given eigenvalues are of multiplicity 2.

\section{Preliminaries and Proof of the results}\label{asýl}

We shall consider only the periodic (for even $n$) eigenvalues of Hill operator. The anti-periodic (for odd $n$)
problem is completely similar. It is well known \cite [Theorem 4.2.3]{Eastham} that the
periodic eigenvalues $\lambda_{2m+1}, \lambda_{2m+2}$ are
asymptotically located in pairs such that
\begin{equation}\label{asy2}
    \lambda_{2m+1}=\lambda_{2m+2}+o(1)=(2m+2)^{2}\pi^{2}+o(1)
\end{equation}
for sufficiently large $m$. From this formula, for all $k\neq 0,(2m+2)$ and  $k\in \mathbb{Z}$, the inequality
\begin{equation}  \label{dist1}
|\lambda-(2(m-k)+2)^{2}\pi^{2}|>|k||(2m+2)-k|>C\,m,
\end{equation}
is satisfied by both eigenvalues $\lambda_{2m+1}$ and $\lambda_{2m+2}$ for large $m$, where, here and in the rest relations, $C
$ denotes a positive constant whose exact value is not essential. Note that when $q=0$, the system $\{e^{-i(2m+2)\pi x}, e^{i(2m+2)\pi x}\}$ is a basis of the
eigenspace corresponding to the double eigenvalues
$(2m+2)^{2}\pi^{2}$ of the problem (\ref{1})-(\ref{per.bc}).

To obtain the asymptotic formulas for the periodic eigenvalues
$\lambda_{2m+1}, \lambda_{2m+2}$ corresponding respectively to the normalized
eigenfunctions $\Psi_{m,1}(x),\Psi_{m,2}(x)$, let us consider the the well-known relation, for sufficiently large
$m$,
\begin{equation}  \label{m1}
\Lambda_{m,j,m-k}(\Psi_{m,j},e^{i(2(m-k)+2)\pi
x})=(q\,\Psi_{m,j},e^{i(2(m-k)+2)\pi x}),
\end{equation}
where $\Lambda_{m,j,m-k}=(\lambda_{2m+j}-(2(m-k)+2)^{2}\pi^{2})$, $j=1,2.$
The relation (\ref{m1}) can be obtained from the equation (\ref{1}), first, replacing $y$ by $\Psi_{m,j}(x),$
 and secondly, multiplying both sides by $e^{i(2(m-k)+2)\pi x}$. By using Lemma 1 in \cite{Melda.O}, to iterate (\ref{m1}) for $k=0$, in the right hand-side of formula (\ref{m1}) we
use the following relations
\begin{equation}  \label{m2}
(q\,\Psi_{m,j},e^{i (2m+2)\pi
x})=\sum_{m_{1}=-\infty}^{\infty}c_{m_1}(\Psi_{m,j},e^{i
(2(m-m_{1})+2)\pi x}),
\end{equation}
\begin{equation}  \label{m3}
|(q\,\Psi_{m,j},e^{i (2(m-m_{1})+2)\pi
x})|< 3M
\end{equation}
for all large $m$, where $j=1,2$ and $M=\sup_{m\in \mathbb{Z}}|c_{m}|$.

First, we fix the terms with indices $m_{1}=0,(2m+2)$. Then all the other terms in the right hand-side of (\ref{m2}) are replaced, in view of (\ref{dist1}) and (\ref{m1}) for $k=m_{1}$, by
\[
c_{m_{1}}\frac{(q\,\Psi_{m,j},e^{i(2(m-m_{1})+2)\pi x})}{\Lambda_{m,j,m-m_{1}}}.
\]

In the same way, by applying the above procedure for the other eigenfunction $e^{-i(2m+2)\pi
x}$ corresponding to the eigenvalue $(2m+2)^{2}\pi^{2}$ of
the problem (\ref{1})-(\ref{per.bc}) for $q=0$, we obtain the following lemma (see also Section 2 in \cite{kýrac:abstract, kýrac:titch}).
\begin{lemma}\label{lemmaitera}
The following relations hold for sufficiently large $m$:
\begin{equation}\label{m4123}
\textrm{(i)}\quad[\Lambda_{m,j,m}- c_{0}-\sum_{i=1}^{2}a_{i}(\lambda_{2m+j})]u_{m,j}=[c_{2m+2}+\sum_{i=1}^{2}b_{i}(\lambda_{2m+j})]v_{m,j}+R_{2},
\end{equation}
where $j=1,2,$
\[
u_{m,j}=(\Psi_{m,j},e^{i(2m+2)\pi x}),\quad v_{m,j}=(\Psi_{m,j},e^{-i(2m+2)\pi x}),
\]
\begin{equation}\label{a1a2}
a_{1}(\lambda_{2m+j})=\sum_{m_{1}}\frac{c_{m_{1}}c_{-m_{1}}}{\Lambda_{m,j,m-m_{1}}},\,
a_{2}(\lambda_{2m+j})=\sum_{m_{1},m_{2}}\frac{c_{m_{1}}c_{m_{2}}c_{-m_{1}-m_{2}}}{\Lambda_{m,j,m-m_{1}}\,\Lambda_{m,j,m-m_{1}-m_{2}}},
\end{equation}
\[
b_{1}(\lambda_{2m+j})=\sum_{m_{1}}\frac{c_{m_{1}}c_{2m+2-m_{1}}}{\Lambda_{m,j,m-m_{1}}},\qquad
b_{2}(\lambda_{2m+j})=\sum_{m_{1},m_{2}}\frac{c_{m_{1}}c_{m_{2}}c_{2m+2-m_{1}-m_{2}}}{\Lambda_{m,j,m-m_{1}}\,\Lambda_{m,j,m-m_{1}-m_{2}}},
\]
\begin{equation}\label{R}
R_{2}=\sum_{m_{1},m_{2},m_{3}}\frac{c_{m_{1}}c_{m_{2}}c_{m_{3}}(q\,\Psi_{m,j}(x),e^{i(2(m-m_{1}-m_{2}-m_{3})+2)\pi x})}{\Lambda_{m,j,m-m_{1}}\,\Lambda_{m,j,m-m_{1}-m_{2}}\,\Lambda_{m,j,m-m_{1}-m_{2}-m_{3}}}.
\end{equation}
The sums in these formulas are taken over all integers $m_{1},m_{2}, m_{3}$ such that $m_{1}, m_{1}+m_{2}, m_{1}+m_{2}+m_{3}\neq 0,\,2m+2$.

\begin{equation}\label{m412}
\textrm{(ii)}\quad[\Lambda_{m,j,m}-c_{0}-\sum_{i=1}^{2} a'_{i}(\lambda_{2m+j})]v_{m,j}=[c_{-2m-2}+\sum_{i=1}^{2}b'_{i}(\lambda_{2m+j})]u_{m,j}+R'_{2},
\end{equation}
where $j=1,2,$
\[
a'_{1}(\lambda_{2m+j})=\sum_{m_{1}}\frac{c_{m_{1}}c_{-m_{1}}}{\Lambda_{m,j,m+m_{1}}},\qquad
a'_{2}(\lambda_{2m+j})=\sum_{m_{1},m_{2}}\frac{c_{m_{1}}c_{m_{2}}c_{-m_{1}-m_{2}}}{\Lambda_{m,j,m+m_{1}}\,\Lambda_{m,j,m+m_{1}+m_{2}}},
\]
\[
b'_{1}(\lambda_{2m+j})=\sum_{m_{1}}\frac{c_{m_{1}}c_{-2m-2-m_{1}}}{\Lambda_{m,j,m+m_{1}}},\qquad
b'_{2}(\lambda_{2m+j})=\sum_{m_{1},m_{2}}\frac{c_{m_{1}}c_{m_{2}}c_{-2m-2-m_{1}-m_{2}}}{\Lambda_{m,j,m+m_{1}}\,\Lambda_{m,j,m+m_{1}+m_{2}}},
\]
\begin{equation}\label{R'}
R'_{2}=\sum_{m_{1},m_{2},m_{3}}\frac{c_{m_{1}}c_{m_{2}}c_{m_{3}}(q\,\Psi_{m,j}(x),e^{i(2(m+m_{1}+m_{2}+m_{3})+2)\pi x})}{\Lambda_{m,j,m+m_{1}}\,\Lambda_{m,j,m+m_{1}+m_{2}}\,\Lambda_{m,j,m+m_{1}+m_{2}+m_{3}}}
\end{equation}
and the sums in these formulas are taken over all integers $m_{1},m_{2}, m_{3}$ such that $m_{1}, m_{1}+m_{2}, m_{1}+m_{2}+m_{3}\neq 0,\,-2m-2$.
\end{lemma}
Note that, by substituting respectively $m_{1}=-k_{1}$ for $i=1$ and $m_{1}+m_{2}=-k_{1}$, $m_{2}=k_{2}$ for $i=2$ into the
relations for $a'_{1}(\lambda_{2m+j})$ and $a'_{2}(\lambda_{2m+j})$, we have the equalities
\begin{equation}\label{a1=a1}
a_{i}(\lambda_{2m+j})=a'_{i}(\lambda_{2m+j})\quad \textrm{for $i=1,2$.}
\end{equation}
  Here, using the equality
\[
  \frac{1}{m_{1}(2m+2-m_{1})}=\frac{1}{2m+2}\left(\frac{1}{m_{1}}+\frac{1}{2m+2-m_{1}}\right),
\]
we get the relation
\[
\sum_{m_{1}\neq 0,(2m+2)}\frac{1}{|m_{1}(2m+2-m_{1})|}=O\left(\frac{ln|m|}{m}\right).
\]
This with (\ref{dist1}), (\ref{m1}) and (\ref{m3}) gives the following estimates (see (\ref{R}), (\ref{R'}))
\begin{equation}\label{m45}
R_{2},\,R'_{2}=O\left((\frac{ln| m|}{m})^{3}\right).
\end{equation}
Moreover, in view of (\ref{dist1}), (\ref{m1}) and (\ref{m3}), we get (see also \cite[Theorem 2]{Melda.O}, \cite{kýrac:titch})
\begin{equation}\label{kare}
\sum_{k\in \mathbb{Z};\,k\neq \pm(m+1)}\Big|(\Psi_{m,j},e^{i2k\pi x})\Big|^{2}=O\left(\frac{1}{m^{2}}\right)
\end{equation}

Therefore, the expansion of the normalized eigenfunctions $\Psi_{m,j}(x)$ by the orthonormal basis $\{e^{i2k\pi x}:k\in \mathbb{Z}\}$ on $[0,1]$ has the following
form
\begin{equation}\label{m7}
\Psi_{m,j}(x)=u_{m,j}\,e^{i(2m+2)\pi x}+v_{m,j}\,e^{-i(2m+2)\pi x}+h_{m}(x),
\end{equation}
where
\[
 \!\!(h_{m},e^{\mp i(2m+2)\pi x})=0,\, \|h_{m}\|=O(m^{-1}),\, \sup_{x\in[0,1]}|h_{m}(x)|=O\left(\frac{ln|m|}{m}\right)
\]
\begin{equation}\label{m8}
  |u_{m,j}|^{2}+|v_{m,j}|^{2}=1+O\left(m^{-2}\right).
\end{equation}

\subsection*{Proof of Theorem \ref{lemma}}

First we estimate the terms of (\ref{m4123}) and (\ref{m412}).
From (\ref{asy2}), (\ref{dist1}) and (\ref{kare}), one can readily see that
\[
\sum_{m_{1}\neq 0,\pm(2m+2)}\left|\frac{1}{\Lambda_{m,j,m\mp m_{1}}}-\frac{1}{\Lambda_{m,0,m\mp m_{1}}}\right|
\]
\begin{equation}\label{dif}
\leq C|\Lambda_{m,j,m}|\sum_{m_{1}\neq 0,\pm(2m+2)}|m_{1}|^{-2}|2m+2\mp m_{1}|^{-2}=o\left(m^{-2}\right),
\end{equation}
where $\Lambda_{m,0,m\mp m_{1}}=((2m+2)^{2}\pi^{2}-(2(m\mp m_{1})+2)^{2}\pi^{2})$. Thus, we get
\begin{equation}\label{a1}
    a_{i}(\lambda_{2m+j})=a_{i}((2m+2)^{2}\pi^{2})+o\left(m^{-2}\right)\quad\textrm{for $i=1,2.$}
\end{equation}
Here, by virtue of (\ref{dif}) we also have, arguing as in \cite[Lemma 3]{kýrac:titch}(see also Lemma 6 of \cite{Veliev:Shkalikov}),
\[b_{1}(\lambda_{2m+j})=\frac{1}{4\pi^{2}}\sum_{m_{1}\neq 0,(2m+2)}\frac{c_{m_{1}}c_{2m+2-m_{1}}}{m_{1}(2m+2-m_{1})}+o\left(m^{-2}\right)\]
\[
=-\int_{0}^{1}(Q(x)-Q_{0})^{2}\,e^{-i2(2m+2)\pi x}dx+o\left(m^{-2}\right)
\]
\begin{equation}\label{I0}
\qquad\;\;=\frac{-1}{i2\pi(2m+2)}\int_{0}^{1}2(Q(x)-Q_{0})\,q(x)\,e^{-i2(2m+2)\pi x}dx+o\left(m^{-2}\right),
\end{equation}
where  \begin{equation}\label{Q0}
Q(x)-Q_{0}=\sum_{m_{1}\neq 0}Q_{m_{1}}\,e^{i2m_{1}\pi x}
\end{equation}
and
$ Q_{m_{1}}=:(Q(x),e^{i2m_{1}\pi x})=\frac{c_{m_{1}}}{i2\pi m_{1}} $ for $m_{1}\neq 0$ are the Fourier coefficients with respect to the system $\{e^{i2m_{1}\pi x}: m_{1}\in\mathbb{Z}\}$ of the function
$Q(x)=\displaystyle\int_{0}^{x}q(t)\, dt.$
Here only for the proof of Theorem \ref{lemma}, we may suppose without loss of generality that  $c_{0}=0$, so that $Q(1)=c_{0}=0.$

Now using the assumption $l_{n}=o(n^{-2})$ of the theorem, it is also $O(n^{-2})$.  In view of (\ref{coskun}) we get $c_{n}=O(n^{-2})$ as $n\rightarrow\infty$. Thus, from Lemma 5 of \cite{Hochstadt:determination}, we obtain that $q(x)$ is absolutely continuous a.e. Hence, for the right hand-side of $b_{1}(\lambda_{2m+j})$ given by (\ref{I0}), integration by parts with $Q(1)=0$ gives
\[
b_{1}(\lambda_{2m+j})=\frac{1}{2\pi^{2}(2m+2)^{2}}\int_{0}^{1}\left(q^{2}(x)+(Q(x)-Q_{0})q^{\prime}(x)\right)e^{-i2(2m+2)\pi x}dx+o\left(m^{-2}\right).
\]
Since $q(x)$ is absolutely continuous a.e., this leads to $\left(q^{2}(x)+(Q(x)-Q_{0})q^{\prime}(x)\right)\in L^{1}[0,1]$. By the Riemann-Lebesgue lemma, we find
\begin{equation}\label{b1o}
b_{1}(\lambda_{2m+j})=o\left(m^{-2}\right).
\end{equation}
Similarly
\begin{equation}\label{b1oprime}
b'_{1}(\lambda_{2m+j})=o\left(m^{-2}\right).
\end{equation}
Let us prove that
\begin{equation}\label{b2oandprime}
b_{2}(\lambda_{2m+j}),\,b'_{2}(\lambda_{2m+j})=o\left(m^{-2}\right).
\end{equation}
Taking into account that $q(x)$ is absolutely continuous a.e. and periodic, we get $c_{m_{1}}c_{m_{2}}c_{\pm(2m+2)-m_{1}-m_{2}}=o\left(m^{-1}\right)$ (see p. 665 of \cite{Veliev:Shkalikov}).
Using this and arguing as in (\ref{m45})
\[
|b_{2}(\lambda_{2m+j})|=o\left(m^{-1}\right)\sum_{m_{1},m_{2}}\frac{1}{\left|m_{1}(2m+2-m_{1})(m_{1}+m_{2})(2m+2-m_{1}-m_{2})\right|}
\]
\[
\!\!\!\!\!\!\!\!\!\!\!\!\!\!\!\!\!\!\!\!\!\!\!\!\!\!\!\!\!\!\!\!\!\!\!\!\!\!\!\!\!\!\!\!\!\!=o\left(m^{-1}\right)O\left((\frac{ln| m|}{m})^{2}\right)=o\left(m^{-2}\right).
\]
Thus, we get the first estimate of (\ref{b2oandprime}). Similarly $b'_{2}(\lambda_{2m+j})=o\left(m^{-2}\right).$
Substituting the estimates given by (\ref{a1=a1}), (\ref{m45}), (\ref{a1}) and (\ref{b1o})-(\ref{b2oandprime}) into the relations (\ref{m4123}) and (\ref{m412}), we find that
\begin{equation}\label{son}
[\Lambda_{m,j,m}-\sum_{i=1}^{2}a_{i}((2m+2)^{2}\pi^{2})]u_{m,j}=c_{2m+2}v_{m,j}+o\left(m^{-2}\right),
\end{equation}
\begin{equation}\label{son'}
[\Lambda_{m,j,m}-\sum_{i=1}^{2}a_{i}((2m+2)^{2}\pi^{2})]v_{m,j}=c_{-2m-2}\,u_{m,j}+o\left(m^{-2}\right)
\end{equation}
for $j=1,2$.

Now suppose that, contrary to what we want to prove, there exists an increasing sequence $\{m_{k}\}\,(k=1,2,\ldots)$ such that
\begin{equation}\label{mk}
|c_{2m_{k}+2}|>C m_{k}^{-2}\quad\textrm{for some $C>0$}.
\end{equation}
Further, the formula obtained from (\ref{m8}) by replacing $m$ with $m_{k}$ shows that either $|u_{m_{k},j}|>1/2$ or $|v_{m_{k},j}|>1/2$ for large $m_{k}$. Without loss of generality we assume that $|u_{m_{k},j}|>1/2$.
Then it follows from both (\ref{son}) and (\ref{son'}) for $m=m_{k}$ that
\begin{equation}\label{sameo}
[\Lambda_{m_{k},j,m_{k}}-\sum_{i=1}^{2}a_{i}((2m_{k}+2)^{2}\pi^{2})]\sim c_{2m_{k}+2},
\end{equation}
where the notation $a_{m}\sim b_{m}$ means that there exist constants $c_{1}$, $c_{2}$ such that $0<c_{1}<c_{2}$ and $c_{1}<|a_{m}/b_{m}|<c_{2}$ for all sufficiently large $m$.
This with (\ref{son'}) for $m=m_{k}$, (\ref{mk}) and the assumption $|u_{m_{k},j}|>1/2$ implies that
\begin{equation}\label{vsimilar}
u_{m_{k},j}\sim v_{m_{k},j}\sim 1.
\end{equation}
Now multiplying (\ref{son'}) for $m=m_{k}$ by $c_{2m_{k}+2}$, and then using (\ref{son}) in (\ref{son'}) for $m=m_{k}$, we arrive at the relation
\[
[\Lambda_{m_{k},j,m_{k}}-\sum_{i=1}^{2}a_{i}((2m_{k}+2)^{2}\pi^{2})]\left([\Lambda_{m_{k},j,m_{k}}-\sum_{i=1}^{2}a_{i}((2m_{k}+2)^{2}\pi^{2})]u_{m_{k},j}+o\left(m_{k}^{-2}\right)\right)
\]\[
=|c_{2m_{k}+2}|^{2}\,u_{m_{k},j}+c_{2m_{k}+2}\,o\left(m_{k}^{-2}\right)
\]
which, by (\ref{sameo}) and (\ref{vsimilar}), implies the following equations
\begin{equation}\label{final}
\Lambda_{m_{k},j,m_{k}}-\sum_{i=1}^{2}a_{i}((2m_{k}+2)^{2}\pi^{2})=\pm|c_{2m_{k}+2}|+o\left(m_{k}^{-2}\right)
\end{equation}
for $j=1,2$.

Let us prove that the periodic eigenvalues for large $m_{k}$ are simple. Assume that there exist two orthogonal eigenfunctions $\Psi_{m_{k},1}(x)$ and $\Psi_{m_{k},2}(x)$ corresponding to $\lambda_{2m_{k}+1}=\lambda_{2m_{k}+2}$. From the argument of Lemma 4 in \cite{Veliev:Shkalikov}, using  the relation (\ref{m7}) with $\|h_{m_{k}}\|=O(m_{k}^{-1})$ for the eigenfunctions $\Psi_{m_{k},j}(x)$ and the orthogonality of eigenfunctions, one can choose the eigenfunction $\Psi_{m_{k},j}(x)$ such that either $u_{m_{k},j}=0$ or $v_{m_{k},j}=0$,
which contradicts (\ref{vsimilar}).

Since the eigenfunctions $\Psi_{m_{k},1}$ and $\overline{\Psi_{m_{k},2}}$  of the self-adjoint problem corresponding to
the different eigenvalues $\lambda_{2m_{k}+1}\neq\lambda_{2m_{k}+2}$ are orthogonal we find, by (\ref{m7}), that
\begin{equation}\label{orteigen}
0=(\Psi_{m_{k},1},\overline{\Psi_{m_{k},2}})=u_{m_{k},2}v_{m_{k},1}+u_{m_{k},1}v_{m_{k},2}+O(m_{k}^{-1}).
\end{equation}
Note that for the simple eigenvalues  in (\ref{final}) there are two cases. First case: The simple eigenvalues $\lambda_{2m_{k}+1}$ and $\lambda_{2m_{k}+2}$ in (\ref{final}) corresponds respectively
to the lower sign $-$ and upper sign $+$. Then
\[
l_{2{m_{k}}+2}=\lambda_{m_{k},2,m_{k}}-\lambda_{m_{k},1,m_{k}}=2|c_{2m_{k}+2}|+o\left(m_{k}^{-2}\right)
\]
which implies that (see (\ref{mk})) $l_{2{m_{k}}+2}>C m_{k}^{-2}$ for some $C$, which contradicts
the hypothesis. Now let us consider the second case: We assume that both  simple eigenvalues correspond
to the lower sign $-$ (the proof for the sign $+$ is similar). Then $\Lambda_{m_{k},2,m_{k}}-\Lambda_{m_{k},1,m_{k}}=o\left(m_{k}^{-2}\right)$. Using this, (\ref{son}) and (\ref{final}), we have
\begin{equation}\label{fark1}
\!\!\!\!\!\!\!\!\!\!\!\!\!\!o\left(m_{k}^{-2}\right)\,u_{m_{k},2}=c_{2m_{k}+2}\,v_{m_{k},2}+|c_{2m_{k}+2}|\,u_{m_{k},2}+o\left(m_{k}^{-2}\right),
\end{equation}
\begin{equation}\label{fark2}
o\left(m_{k}^{-2}\right)\,u_{m_{k},1}=-c_{2m_{k}+2}\,v_{m_{k},1}-|c_{2m_{k}+2}|\,u_{m_{k},1}+o\left(m_{k}^{-2}\right).
\end{equation}
Therefore, multiplying both sides of (\ref{fark1}) and (\ref{fark2}) by $v_{m_{k},1}$ and $v_{m_{k},2}$,
respectively, and adding the two resulting relations, we have, in view of (\ref{mk}),
\[
u_{m_{k},2}v_{m_{k},1}-u_{m_{k},1}v_{m_{k},2}=o(1).
\]
This with (\ref{orteigen}) gives $u_{m_{k},2}v_{m_{k},1}=o(1)$ which contradicts (\ref{vsimilar}). Thus the assumption (\ref{mk}) is false, that is, $c_{2m+2}=o\left(m^{-2}\right)$. A similar result holds for the anti-periodic problem, that is, $c_{2m+1}=o\left(m^{-2}\right)$. The theorem is proved.
$\square$

For the proof of Theorem \ref{main0} we need the sharper estimates of the following lemma:
\begin{lemma}\label{lemmaa12}
Let $q(x)$ be absolutely continuous a.e. and $c_{0}=0$. Then, for all sufficiently large $m$, we have the following equalities for the series in  (\ref{a1a2})
\begin{equation}\label{lemma1}
a_{1}(\lambda_{2m+j})=\frac{-1}{(2\pi(2m+2))^{2}}\int_{0}^{1}q^{2}(x)dx+o\left(m^{-2}\right),\quad
a_{2}(\lambda_{2m+j})=o\left(m^{-2}\right).
\end{equation}
\end{lemma}
\begin{proof}
First, let us consider $a_{1}(\lambda_{2m+j})$. By virtue of (\ref{dif}) we get
\[
    a_{1}(\lambda_{2m+j})=\frac{1}{4\pi^{2}}\sum_{m_{1}\neq 0,(2m+2)}\frac{c_{m_{1}}c_{-m_{1}}}{m_{1}(2m+2-m_{1})}+o\left(m^{-2}\right).
\]
Arguing as in Lemma 3 in \cite{kýrac:titch} (see also Lemma 2.3(a) of \cite{veliev;arþiv}), we obtain, in our notations,
\[
 a_{1}(\lambda_{2m+j})=\frac{1}{2\pi^{2}}\sum_{m_{1}> 0,m_{1}\neq(2m+2)}\frac{c_{m_{1}}c_{-m_{1}}}{(2m+2+m_{1})(2m+2-m_{1})}+o\left(m^{-2}\right)
\]
\[
=\int_{0}^{1}(G^{+}(x,m)-G^{+}_{0}(m))^{2}\,e^{i2(4m+4)\pi x}\,dx+o\left(m^{-2}\right)=
\]
\begin{equation}\label{d3}
\frac{-2}{i2\pi(4m+4)}\int_{0}^{1}(G^{+}(x,m)-G^{+}_{0}(m))(q(x)e^{-i2(2m+2)\pi x}-c_{2m+2})e^{i2(4m+4)\pi x}dx+o\left(m^{-2}\right)
\end{equation}
where
\begin{equation}\label{d4}
G^{\pm}_{m_{1}}(m)=:(G^{\pm}(x,m), e^{i2m_{1}\pi x})=\frac{c_{m_{1}\pm(2m+2)}}{i2\pi m_{1}}
\end{equation}
for $m_{1}\neq 0$ are the Fourier coefficients with respect to $\{e^{i2m_{1}\pi x}: m_{1}\in\mathbb{Z}\}$ of the functions
\begin{equation}\label{d2}
G^{\pm}(x,m)=\int_{0}^{x}q(t)\,e^{\mp i2(2m+2)\pi t}dt-c_{\pm(2m+2)}x
\end{equation}
and
\[
G^{\pm}(x,m)-G^{\pm}_{0}(m)=\sum_{m_{1}\neq(2m+2)}\frac{c_{m_{1}}}{i2\pi(m_{1}\mp(2m+2))}\,e^{i2(m_{1}\mp(2m+2))\pi x}.
\]
Here, taking into account the Lemma 1 of \cite{kýrac:titch} and (\ref{d2}), we have the estimates
\begin{equation}\label{ggg}
G^{\pm}(x,m)-G^{\pm}_{0}(m)=G^{\pm}(x,m)-\int_{0}^{1}G^{\pm}(x,m)\, dx=o(1)\quad\textrm{as $m\rightarrow\infty$}
\end{equation}
uniformly in $x$.

From the equalities (see (\ref{d2}))
\begin{equation}\label{gg}
    G^{\pm}(1,m)=G^{\pm}(0,m)=0
\end{equation}
and since $q(x)$ is absolutely continuous a.e.,  integration by parts gives for the right hand-side of $a_{1}(\lambda_{2m+j})$ given by (\ref{d3})  the value
\[
 a_{1}(\lambda_{2m+j})=\frac{-1}{(2\pi(2m+2))^{2}}\left[\int_{0}^{1}q^{2}+\int_{0}^{1}(G^{+}(x,m)-G^{+}_{0}(m))q^{\prime}(x)e^{i2(2m+2)\pi x}dx\right]
\]
\[
+\frac{|c_{2m+2}|^{2}}{(2\pi(2m+2))^{2}}+o\left(m^{-2}\right)
\]
for sufficiently large $m$. Thus, by using the Riemann-Lebesgue lemma,  this with $(G^{+}(x,m)-G^{+}_{0}(m))q^{\prime}(x)\in L^{1}[0,1]$ implies the first equality of (\ref{lemma1}).

Now, it remains to prove that $a_{2}(\lambda_{2m+j})=o\left(m^{-2}\right)$. Similarly, by (\ref{a1}) for $i=2$ we get
\begin{equation}\label{a2}
a_{2}(\lambda_{2m+j})=\sum_{m_{1},m_{2}}\frac{(2\pi)^{-4}\,c_{m_{1}}c_{m_{2}}c_{-m_{1}-m_{2}}}{m_{1}(2m+2-m_{1})(m_{1}+m_{2})(2m+2-m_{1}-m_{2})}+o\left(m^{-2}\right).
\end{equation}
As in Lemma 4 of \cite{kýrac:titch}, using the summation variable $m_{2}$ to represent the previous $m_{1}+m_{2}$ in (\ref{a2}), we write (\ref{a2}) in the form
\[
 a_{2}(\lambda_{2m+j})=\frac{1}{(2\pi)^{4}}\sum_{m_{1},m_{2}}\frac{c_{m_{1}}c_{m_{2}-m_{1}}c_{-m_{2}}}{m_{1}(2m+2-m_{1})m_{2}(2m+2-m_{2})},
\]
where the forbidden indices in the sums take the form of $m_{1}, m_{2}\neq 0,\,2m+2$.
 Here the equality
\[
  \frac{1}{k(2m+2-k)}=\frac{1}{2m+2}\left(\frac{1}{k}+\frac{1}{2m+2-k}\right)
\]
gives
\begin{equation}\label{equal1}
  a_{2}(\lambda_{2m+j})=\frac{1}{(2\pi)^{4}(2m+2)^{2}}\sum_{j=1}^{4}S_{j},
\end{equation}
where
\[
S_{1}=\sum_{m_{1},m_{2}}\frac{c_{m_{1}}c_{m_{2}-m_{1}}c_{-m_{2}}}{m_{1}m_{2}},\; S_{2}=\sum_{m_{1},m_{2}}\frac{c_{m_{1}}c_{m_{2}-m_{1}}c_{-m_{2}}}{m_{2}(2m+2-m_{1})},
\]
\[
S_{3}=\sum_{m_{1},m_{2}}\frac{c_{m_{1}}c_{m_{2}-m_{1}}c_{-m_{2}}}{m_{1}(2m+2-m_{2})},\; S_{4}=\sum_{m_{1},m_{2}}\frac{c_{m_{1}}c_{m_{2}-m_{1}}c_{-m_{2}}}{(2m+2-m_{1})(2m+2-m_{2})}.
\]
From (\ref{Q0}) and the assumption $c_{0}=0$ which implies $Q(1)=0$, we deduce by means of the substitution $t=(Q(x)-Q_{0})$
\begin{equation}\label{s11}
 S_{1}=4\pi^{2}\int_{0}^{1}(Q(x)-Q_{0})^{2}q(x)\,dx=0.
\end{equation}
Similarly, in view of (\ref{Q0}) and (\ref{d4})-(\ref{gg}), we get by the Riemann-Lebesgue lemma
\[
S_{2}=-4\pi^{2}\int_{0}^{1}(Q(x)-Q_{0})(G^{+}(x,m)-G^{+}_{0}(m))\,q(x)\,e^{i2(2m+2)\pi x}dx=o\left(1\right),
 \]
\[
 S_{3}=-4\pi^{2}\int_{0}^{1}(Q(x)-Q_{0})(G^{-}(x,m)-G^{-}_{0}(m))\,q(x)\,e^{-i2(2m+2)\pi x}dx=o\left(1\right)
\]
and by (\ref{ggg})
\[
S_{4}=4\pi^{2}\int_{0}^{1}(G^{+}(x,m)-G^{+}_{0}(m))(G^{-}(x,m)-G^{-}_{0}(m))\,q(x)\,dx=o\left(1\right).
\]
Thus, these with (\ref{equal1}) and (\ref{s11}) imply that $a_{2}(\lambda_{2m+j})=o\left(m^{-2}\right)$. The lemma is proved.
\end{proof}

\subsection*{Proof of Theorem \ref{main0}}
(i) First let us prove that $c_{0}=0$. By considering the first step of the procedure in the Lemma \ref{lemmaitera} and using a similar estimate as in (\ref{m45}), we may rewrite the relations (\ref{m4123}) and (\ref{m412}) as follows:
\begin{equation}\label{sonnnnn}
\left.
\begin{array}{ll}
 \displaystyle [\Lambda_{m,j,m}- c_{0}]u_{m,j}=c_{2m+2}v_{m,j}+O\left(\frac{ln| m|}{m}\right), & \\\\
\displaystyle [\Lambda_{m,j,m}- c_{0}]v_{m,j}=c_{-2m-2}\,u_{m,j}+O\left(\frac{ln| m|}{m}\right) &  \\
\end{array}
\!\!\!\!\!\!\!\right\}
\end{equation}
for $j=1,2$ and sufficiently large $m$. By using the assumption $l_{2m+2}=o(m^{-2})$, namely, $l_{n}=o(n^{-2})$ for even $n=2m+2$  and Theorem \ref{lemma} which implies $c_{\mp(2m+2)}=o(m^{-2})$, we obtain the relations (see (\ref{sonnnnn})) in the form
\begin{equation}\label{c01}
[\Lambda_{m,j,m}- c_{0}]u_{m,j}=O\left(\frac{ln| m|}{m}\right),
\end{equation}
\begin{equation}\label{c02}
[\Lambda_{m,j,m}- c_{0}]v_{m,j}=O\left(\frac{ln| m|}{m}\right).
\end{equation}
Again by (\ref{m8}) we have, for large $m$, either $|u_{m,j}|>1/2$ or $|v_{m,j}|>1/2$. In either case, in view of (\ref{c01}) and (\ref{c02}), there exists a positive large number $N_{0}$ such that both the eigenvalues $\lambda_{2m+j}$ (see definition of (\ref{m1})) satisfy the following estimate
\begin{equation}\label{c0son}
\lambda_{2m+j}=(2m+2)^{2}\pi^{2}+c_{0}+O\left(\frac{ln| m|}{m}\right)
\end{equation}
for all $m>N_{0}$ and $j=1,2$. When $m>\max\{(n_{0}-2)/2,N_{0}\}$, from the assumption of Theorem \ref{main0} (i) the eigenvalue $(2m+2)^{2}\pi^{2}$ corresponds to the eigenvalue $\lambda_{2m+1}$ or $\lambda_{2m+2}$. In either case we obtain $c_{0}=0$ by (\ref{c0son}).

Finally, for sufficiently large $m$, substituting the estimates of
 \[a_{i}(\lambda_{2m+j}),\,a'_{i}(\lambda_{2m+j}),\,  b_{i}(\lambda_{2m+j}),\,  b'_{i}(\lambda_{2m+j}),\, R_{2},\, R'_{2}\]
for $i=1,2$, respectively, given by Lemma \ref{lemmaa12} with the equalities $a_{i}(\lambda_{2m+j})=a'_{i}(\lambda_{2m+j})$ (see (\ref{a1=a1})), (\ref{b1o})-(\ref{b2oandprime}) and (\ref{m45}) in the relations (\ref{m4123}) and (\ref{m412}) and using $c_{0}=0$, we find the relations in the following form
\begin{equation}\label{q2son}
\left.
\begin{array}{ll}
 \displaystyle \left[\Lambda_{m,j,m}+\frac{1}{(2\pi(2m+2))^{2}}\int_{0}^{1}q^{2}\right]u_{m,j}=c_{2m+2}v_{m,j}+o\left(m^{-2}\right), & \\\\
\displaystyle \left[\Lambda_{m,j,m}+\frac{1}{(2\pi(2m+2))^{2}}\int_{0}^{1}q^{2}\right]v_{m,j}=c_{-2m-2}\,u_{m,j}+o\left(m^{-2}\right) &  \\
\end{array}
\!\!\!\!\!\!\!\right\}
\end{equation}
for $j=1,2$.
In the same way, by using the assumption $l_{2m+2}=o(m^{-2})$ and Theorem \ref{lemma}, we write (\ref{q2son}) in the form
\[
\left[\Lambda_{m,j,m}+\frac{1}{(2\pi(2m+2))^{2}}\int_{0}^{1}q^{2}\right]u_{m,j}=o\left(m^{-2}\right),
\]
\[
\left[\Lambda_{m,j,m}+\frac{1}{(2\pi(2m+2))^{2}}\int_{0}^{1}q^{2}\right]v_{m,j}=o\left(m^{-2}\right).
\]
Thus, arguing as in the proof of (\ref{c0son}), there exists a positive large number $N_{1}$ such that the eigenvalues $\lambda_{2m+j}$ satisfy the following estimate
\begin{equation}\label{sssson}
\lambda_{2m+j}=(2m+2)^{2}\pi^{2}-\frac{1}{(2\pi(2m+2))^{2}}\int_{0}^{1}q^{2}+o\left(m^{-2}\right)
\end{equation}
for all $m>N_{1}$ and $j=1,2$. Let $m>\max\{(n_{0}-2)/2,N_{1}\}$. Using the same
argument as above, by (\ref{sssson}), we get $\int_{0}^{1}q^{2}=0$ which implies that $q=0$ a.e.

(ii) The same argument in Section \ref{asýl} works  for the anti-periodic boundary conditions
\begin{equation*}\label{antiper}
\qquad\qquad y(0)=-y(a),\qquad y^{\prime}(0)=-y^{\prime}(a)
\end{equation*}
and one can readily see the corresponding results for the anti-periodic eigenvalues $\mu_{2m}$, $\mu_{2m+1}$ from (\ref{asy2}), (\ref{dist1}) and (\ref{m1}) by replacing $2m+2$ with $2m+1$. Then, arguing as in the proof of Theorem \ref{main0} (i), we get the assertion of Theorem \ref{main0} (ii).
$\square$

\end{document}